\documentclass{amsart}
\usepackage[pdftex]{graphicx}

\usepackage{mathtools}
\usepackage{afterpage}
\usepackage{enumitem}
\usepackage{multirow}

\usepackage{bm}
\usepackage{color}
\usepackage{caption} 
\mathtoolsset{showonlyrefs=true}

\numberwithin{equation}{section}

\newtheorem{theorem}{Theorem}[section]
\newtheorem{lemma}[theorem]{Lemma}

\newtheorem{proposition}[theorem]{Proposition}

\theoremstyle{definition}
\newtheorem{definition}[theorem]{Definition}

\newtheorem*{acknowledgment}{Acknowledgments}

\theoremstyle{remark}
\newtheorem{remark}[theorem]{Remark}

\numberwithin{equation}{section}

\newcommand{\Z}{\mathbb{Z}}
\newcommand{\Ord}{{\rm Ord}}
\newcommand{\HFK}{{\rm HFK}}
\newcommand{\CFK}{{\rm CFK}}
\newcommand{\F}{\mathbb{F}}

\newcommand{\bridge}{{\rm bridge}}
\newcommand{\braid}{{\rm braid}}

\begin{document}

\title{The bridge index and the braid index for twist positive knots}

\author{Keisuke Himeno
}

\address{Graduate School of Advanced Science and Engineering, Hiroshima University,
1-3-1 Kagamiyama, Higashi-hiroshima, 7398526, Japan}
\email{himeno-keisuke@hiroshima-u.ac.jp}

\begin{abstract}
In general, the bridge index of a knot is less than or equal to its braid index. A natural question is when these two values coincide. Motivated by a conjecture of Krishna and Morton, we prove that the bridge index and the braid index coincide for all twist positive knots, using the knot Floer torsion order. Here, a twist positive knot is a knot that admits a positive braid representative containing at least one full twist.
\end{abstract}

\renewcommand{\thefootnote}{}

\maketitle


\section{Introduction}
For a knot $K$ in $S^3$, its \emph{bridge index} $\bridge(K)$ is a geometric invariant introduced by Schubert \cite{Sch54}. It is defined as the minimal number of arcs when $(S^3,K)$ is decomposed into two trivial tangles. For instance, $\bridge(K)=1$ if and only if $K$ is the unknot, and $\bridge(T(p,q))=\min\{|p|,|q|\}$ where $T(p,q)$ is the $(p,q)$--torus knot. However, determining the bridge index is generally a difficult problem.

A related concept is the \emph{braid index} $\braid(K)$, which is the minimal number of strands required in a braid whose closure is isotopic to $K$. By definition, we always have 
\[
\bridge(K)\le \braid(K).
\]
However, the equality does not hold in general. For example, any hyperbolic two-bridge knot $K$ satisfies $\bridge(K)=2$ and $\braid(K)\ge3$.

In \cite{KM25}, Krishna and Morton conjectured that the bridge index and the braid index are equal for two particular classes of knots. One is the class of \emph{positive braid knots}, that is, knots that can be realized as the closure of a positive braid. The other is the class of \emph{$L$--space knots}. An $L$--space knot is a knot that admits a positive Dehn surgery yielding an $L$--space (see \cite{OS05}). In their paper, they investigated the positive braid knots that ``contain at least one full twist'', called \emph{twist positive knots}\textup{:} 

\begin{definition}
A positive $n$--braid $\beta$ is called a \emph{twist positive $n$--braid} if $\beta$ can be written as $\beta=\Delta^2\gamma$, where $\Delta^2$ is the full twist on $n$ strands and $\gamma$ is a positive braid. If a knot $K$ is the closure of such a braid, then $K$ is called a \emph{twist positive knot on $n$ strands}.
\end{definition}

In knot theory, twist positive knots often play important roles in various studies. For example, the braid index of a twist positive knot on $n$ strands equals to $n$ \cite{FW87,Mor86}.

Krishna and Morton confirmed their conjecture for twist positive $L$--space knots by computing the Alexander polynomials and applying the \emph{knot Floer torsion order} introduced in \cite{JMZ20}. As a recent research related to this topic, there is a paper that computes the knot Floer torsion order for any $L$--space cabling of any $L$--space knot \cite{Suc25}.

The purpose of this paper is to prove that this conjecture holds for general twist positive knots\textup{:}

\begin{theorem}\label{thm_main}
If $K$ is a twist positive knot on $n$ strands, then $\bridge(K)=\braid(K)=n$. In particular, if $K$ is a positive twisted torus knot $T(p,q;r,s)$ with $0 < r < p < q$ and $s \ge 0$, then $\bridge(K) = p$.
\end{theorem}

The twisted torus knot $T(p,q;r,s)$ is defined to be the $(p, q)$--torus knot with $s$ full twists on $r$ adjacent strands where $0 < r < p$. The bridge index of twisted torus knots was studied in \cite{BTZ15}, where it was shown that the bridge index of a twisted torus knot $T(p,q;r,s)$ with $p < q$ is $p$ for sufficiently large $|s|$. In contrast, Theorem \ref{thm_main} removes the assumption on $s$ for positive twisted torus knots.

To prove Theorem \ref{thm_main}, we carry out a detailed study of the knot Floer torsion order $\Ord(K)\in\mathbb{N}\cup\{0\}$. The following proposition is the key to our proof. Let $\tau(K) \in \mathbb{Z}$ and $\Upsilon_K\colon [0,2] \to \mathbb{R}$ denote the tau invariant and the Upsilon invariant of $K$, respectively (see Section \ref{sec_preliminaries}).
 
 \begin{proposition}\label{prop_main}
Suppose that a knot $K$ satisfies the following two conditions for some positive integer $n$\textup{:}
\begin{itemize}
\item[{\rm (1)}] $\Upsilon_K(t)=-\tau(K)\cdot t$ for $0 \le t \le 2/n$, and
\item[{\rm (2)}] $\Upsilon_K(t)>-\tau(K)\cdot t$ for $2/n < t \le 1$.
\end{itemize}
Then $\Ord(K)\ge n-1$.
 \end{proposition}

The invariant $\Ord(K)$ provides a lower bound for the bridge index \cite{JMZ20}\textup{:}
\[
\Ord(K)\le \bridge(K)-1,
\] 
and hence,
\[
\Ord(K)+1\le \bridge(K)\le \braid(K).
\]
Krishna and Morton showed that the bridge index and the braid index are equal for twist positive $L$--space knots by using this inequality. 

We now prove Theorem \ref{thm_main}, assuming Proposition \ref{prop_main}.

\begin{proof}[Proof of Theorem \ref{thm_main}]
The statement is clear for $n=1,2$, so we may assume that $n\ge 3$. 

By Lemma 3.5 and Proposition 3.7 of \cite{FK17}, any twist positive knot on $n$ strands satisfies conditions (1) and (2) in Proposition \ref{prop_main}. (Note that condition (2) was originally stated in terms of the smooth $4$--genus. However, from the statement for $0 \le t \le 2/n$ in Proposition 3.7 of that paper, it follows that the tau invariant coincides with the smooth $4$--genus.) Thus, $\Ord(K)\ge n-1$.
Therefore, we have
\[
n\le \Ord(K)+1\le \bridge(K)\le \braid(K)\le n,
\]
and hence, $\bridge(K)=\braid(K)=n$.
\end{proof}

Moreover, any knot that can be represented as the closure of an $n$--braid of the form $(\sigma_1\sigma_2\cdots\sigma_{n-1})^{nk+1}\alpha$ for $k \ge 1$, where each $\sigma_i$ is a standard braid generator and $\alpha$ is a quasi-positive braid, satisfies the conditions (1) and (2) in Proposition \ref{prop_main} \cite{FK17}. Therefore, the bridge index and the braid index of such a knot are both equal to $n$.

\begin{acknowledgment}
We would like to thank Masakazu Teragaito for his thoughtful guidance and helpful discussions about this work. The author was supported by JST SPRING, Grant Number JPMJSP2132.
\end{acknowledgment}

\section{Preliminaries}\label{sec_preliminaries}
In this section, we review the knot Floer homology theory and the Upsilon invariant.

\subsection{knot Floer homology}
For clarity of notation and conventions, we briefly recall the facts. For further details, see \cite{Hom17, Man16, OS04}.

For a knot $K \subset S^3$, the knot Floer complex $\CFK^{\infty}(K)$ is a finitely generated $\F_2[U, U^{-1}]$--module, where $\F_2$ denotes the field with two elements and $U$ is a formal variable. It is equipped with a differential $\partial$ and a $\Z \oplus \Z$ filtration with respect to the partial order $\le$, defined by $(i,j)\le (k,l)$ if and only if $i\le k$ and $j\le l$. The differential $\partial$ respects the filtration in the sense that the filtration level of $\partial (x)$ is less than or equal to that of $x$.

When $x$ is a base element of $\CFK^{\infty}(K)$ as an $\F_2$--vector space, we write $[x,i,j]$ to denote that $x$ lies in filtration level $(i,j)$. A \emph{filtered base change} is an operation that produces a new filtered element $[x',i,j]=[x,i,j]+[y,k,l]$, where $(k,l)\le (i,j)$.

The action of $U$ on $\CFK^{\infty}(K)$ commutes with the differential $\partial$, decreases the integer-valued homological grading (called \emph{Maslov grading}) by $2$, and lowers the $(i, j)$ filtration level by $(1, 1)$:
\[
U \cdot [x, i, j] = [Ux, i - 1, j - 1].
\]

The homology of $(\CFK^{\infty}(K),\partial)$ is isomorphic to $\mathrm{HF}^{\infty}(S^3)\cong\F_2[U,U^{-1}]$, where the Maslov grading is normalized so that the generator of the homology lies in Maslov grading zero. So, $H_0(\CFK^{\infty}(K))\cong \F_2$, see, for example, \cite{OSS17,OS04A, OS04}.
Furthermore, $\CFK^{\infty}(K)$ admits a symmetry with respect to the $(i, j)$ filtration: the map $(i, j) \mapsto (j, i)$ induces an isomorphism of complexes.

 The vertical differential $\partial_V$ is the $i$--grading preserving restriction of $\partial$. Specifically, if 
\[
\partial([x,i,j]) = \sum_k [y_k, i_k, j_k],
\]
then
\[
\partial_V([x, i, j]) = \sum_{i_k = i} [y_k, i_k, j_k].
\]
Moreover, then we say that the length of $\partial_V ([x, i, j])$ is \[
j-\max\{j_k\mid \text{$[y_k,i_k,j_k]$ with $i_k=i$} \}.
\]
Similarly, the horizontal differential $\partial_H$ denotes the $j$--grading preserving restriction of $\partial$.

The subcomplex $\CFK^-(K)$ is generated over $\F_2$ by base elements $[x,i,j] \in \CFK^{\infty}(K)$ with $i \le 0$. It is also a finitely generated $\F_2[U]$--module. Let $\HFK^-(K)$ denote the homology of $(\CFK^-(K),\partial_H)$. We denote by $\mathrm{Tor}(\HFK^-(K))$ the $U$-torsion submodule of $\HFK^-(K)$. The \emph{torsion order} of a knot $K$ is defined by
\[
\Ord(K) = \min \left\{ k \ge 0 \,\middle|\, U^k \cdot \mathrm{Tor}(\HFK^-(K)) = 0 \right\},
\]
see \cite{JMZ20}.

The complex $\widehat{\CFK}(K)$ is the $\F_2$--submodule of $\CFK^{\infty}(K)$ generated by base elements $[x, 0, j]$. The homology of this complex with respect to the differential $\widehat{\partial}$---defined as the restriction of $\partial$ to maps preserving the $(i, j)$--grading---is denoted by $\widehat{\HFK}(K)$. $\widehat{\HFK}(K)$ is a bigraded by the Maslov grading and the $j$--grading, the latter being called the \emph{Alexander grading}.

The homology of $(\widehat{\CFK}(K),\partial_V)$ is isomorphic to $\widehat{{\rm HF}}(S^3) \cong \F_2$ (by symmetry, the same holds for $\partial_H$). The generator lies in the Maslov grading $0$. Moreover, its Alexander grading is the \emph{tau invariant} $\tau(K) \in \Z$.

\subsection{The Upsilon invariant}

The \emph{Upsilon invariant} $\Upsilon_K\colon [0,2]\to\mathbb{R}$ of a knot $K\subset S^3$ has properties
\begin{itemize}
\item is invariant under knot concordance,
\item $\Upsilon_K$ is continuous and piecewise linear,
\item $\Upsilon_K(0)=\Upsilon_K(2)=0$
\item $\Upsilon_K(t)=\Upsilon_K(2-t)$ for $t\in[0,2]$,
\item for small $t$, $\Upsilon_K(t)=-\tau(K)\cdot t$,
\end{itemize}
see \cite{Liv17,OSS17}.

We now review Livingston's method for computing the Upsilon invariant \cite{Liv17}.  
Let $C_{t,s}$ ($t \in [0,2],\ s \in \mathbb{R}$) be the subcomplex of $\CFK^\infty(K)$ spanned over $\F_2$ by all base elements $[x, i, j]$ satisfying 
\[
j \cdot \frac{t}{2} + i \cdot \left(1 - \frac{t}{2} \right) \le s.
\]
Note that for each $t\in[0,2]$, $\cup_{s\in\mathbb{R}}C_{t,s}=\CFK^{\infty}(K)$.
Then, the Upsilon invariant is given by
\[
\Upsilon_K(t)=\max\left\{-2s \mid \text{$H_0(C_{t,s})\to H_0(\CFK^\infty(K))\cong \F_2$ is surjective} \right\},
\]
where $H_0(C_{t,s})\to H_0(\CFK^\infty(K))$ is the induced homomorphism by the inclusion.

\section{Proof of Proposition \ref{prop_main}}

We now proceed to  prove Proposition \ref{prop_main}.
For simplicity, we use the term \emph{nontrivial cycle} to refer to a cycle in $\CFK^{\infty}(K)$ representing a nontrivial homology class in $H_0(\CFK^{\infty}(K))$. 

\begin{lemma}\label{lem_1}
Suppose that a knot $K$ satisfies the following two conditions for some positive integer $n$\textup{:}
\begin{itemize}
\item[{\rm (1)}] $\Upsilon_K(t)=-\tau(K)\cdot t$ for $0 \le t \le 2/n$, and
\item[{\rm (2)}] $\Upsilon_K(t)>-\tau(K)\cdot t$ for $2/n < t \le 1$.
\end{itemize}
Then the complex $(\CFK^{\infty}(K),\partial)$ contains a nontrivial cycle $c=\sum_k[c_k, i^c_k, j^c_k]$ such that\textup{:}
\begin{itemize}
\item $j^c_k \le \tau-i^c_k(n-1)$ for all $k$, and
\item $j^c_k = \tau-i^c_k(n-1)$ and $i^c_k\ge 1$ for some $k$.
\end{itemize}
Moreover, there is no nontrivial cycle $d=\sum_k[d_k,i^d_k,j^d_k]$ such that $j^d_k < \tau - i^d_k(n-1)$ for all $k$.
\end{lemma}

\begin{proof}
We first show that no nontrivial cycle $d=\sum_k[d_k,i^d_k,j^d_k]$ exists with $j^d_k < \tau - i^d_k(n-1)$ for all $k$. Suppose, for contradiction, that such a cycle $d$ exists. Consider the parameter $t=2/n$. Then the subcomplex $C_{2/n,s}$ is spanned by elements $[x,i,j]$ satisfying $j\le ns-i(n-1)$.
Then $d\in C_{2/n,s_d}$ for some $s_d<\tau/n$, see Figure \ref{lem_fig_1}. Hence, $\Upsilon_K(2/n)\ge -2s_d > -2\tau/n$, which contradicts to condition (1).

\begin{figure}
\centering
\includegraphics[scale=0.1]{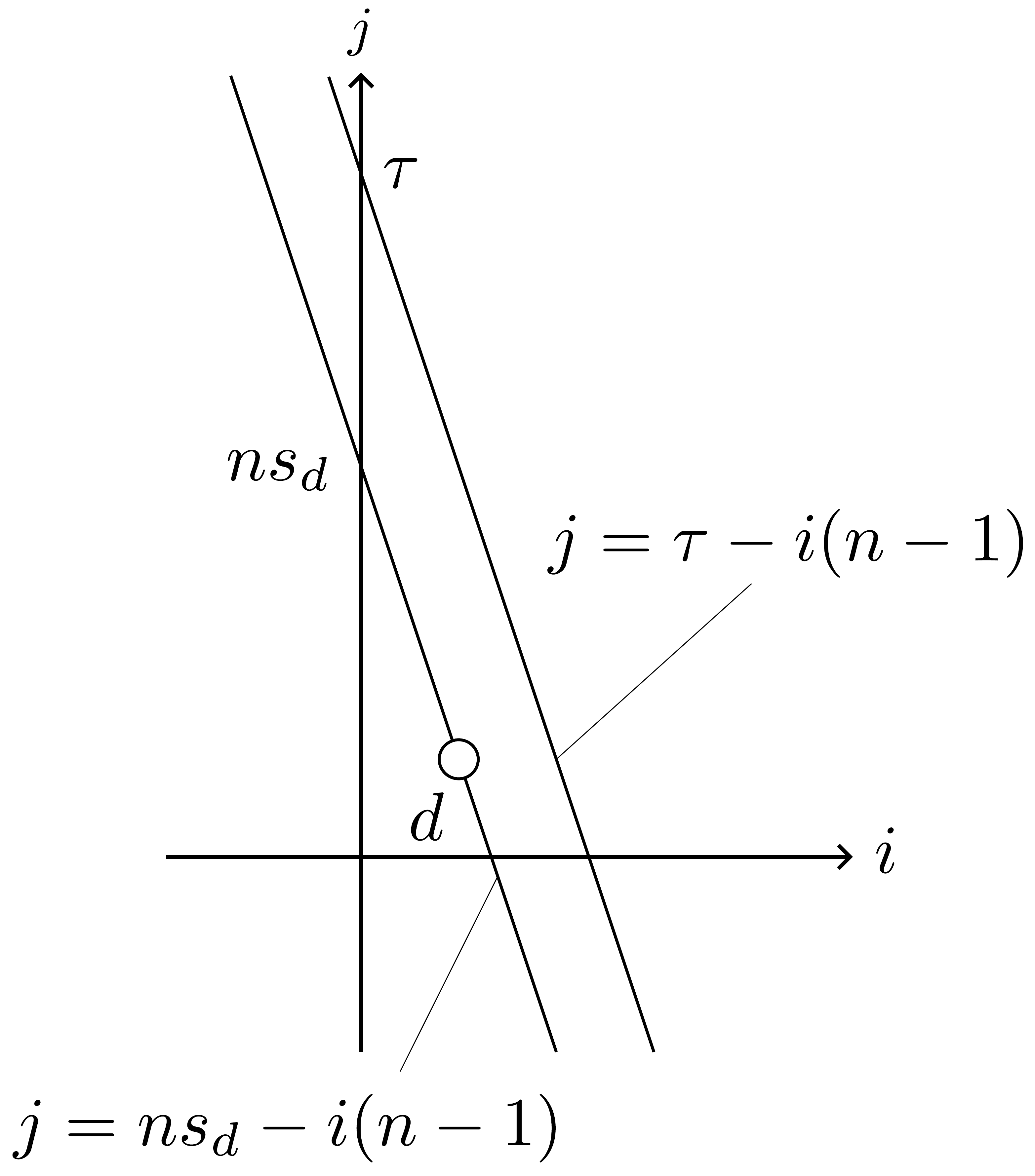}
\caption{The $(i,j)$--plane. If such a cycle $d$ exists, it must be supported strictly below the line $j=\tau-i(n-1)$, meaning that $d\in C_{2/n,s_d}$ for some $s_d<\tau/n$.}
\label{lem_fig_1}
\end{figure}

Finally, suppose there is no nontrivial cycle $c$ as described. Then $\Upsilon_K(t) = -\tau t$ would still hold for values of $t$ slightly greater than $2/n$, contradicting conditions (1) and (2). This completes the proof.
\end{proof}

 \begin{proposition}[Proposition \ref{prop_main} again]\label{prop_again}
Suppose that a knot $K$ satisfies the following two conditions for some positive integer $n$\textup{:}
\begin{itemize}
\item[{\rm (1)}] $\Upsilon_K(t)=-\tau(K)\cdot t$ for $0 \le t \le 2/n$, and
\item[{\rm (2)}] $\Upsilon_K(t)>-\tau(K)\cdot t$ for $2/n < t \le 1$.
\end{itemize}
Then $\Ord(K)\ge n-1$.
 \end{proposition}
\begin{proof}
For simplicity, let $\tau=\tau(K)$. 

By Lemma \ref{lem_1}, there is a nontrivial cycle $c=\sum_{k=1}^l[c_k, i^c_k, j^c_k]$ such that
\begin{itemize}
\item $j^c_k \le \tau-i^c_k(n-1)$ for all $k$, and
\item $j^c_k = \tau-i^c_k(n-1)$ and $i^c_k\ge 1$ for some $k$.
\end{itemize}
Let $m=\max\{i^c_k\mid k=1,\ldots,l\}$, 
and let $x=\sum_{i^c_k=m}c_k$. By applying a filtered base change, we may suppose that $x$ is a base element $[x,m,j^x]$.

Note that $\partial_V (x)=0$, since $c$ is a cycle and $\partial$ respects the filtration level. Moreover, because that $U^mx$ cannot realize $\tau$, $x$ is a boundary cycle for $\partial_V$. Hence, there is a base element $[y,m,j^y]$ such that $\partial_V (y)=x$. 

Assume that the length of $\partial_V (y)$ is less than $n-1$, that is, 
\[
j^y-j^x\le n-2. 
\]
Since $\partial^2 (y)=0$ and $c$ is a nontrivial cycle, $\partial (y)=x+z$ for some nontrivial cycle $z=\sum_k[z_k,i^z_k,j^z_k]$. 

Now, for any $k$, we have
\begin{align*}
j^z_k &\le j^y \\
      &\le j^x + (n - 2) \\
      &\le \tau - m(n - 1) + (n - 2) \\
      &= \tau - (m - 1)(n - 1) - 1 \\
      &\le \tau - i^z_k(n - 1) - 1 \\
      &< \tau - i^z_k(n - 1).
\end{align*}
Here, the first inequality follows from the fact that the differential $\partial$ respects the filtration level.  
The second inequality comes from the length of $\partial_V(y)$ being at most $n-2$.  
The third follows from our choice of $x$.  
The fourth inequality uses that $\partial_V(y) = x$.
This contradicts Lemma~\ref{lem_1}, which asserts that there is no such nontrivial cycle $z$.  
Therefore, the length of $\partial_V(y)$ must be at least $n - 1$ (see also Figure \ref{fig_1}).

\begin{figure}
\centering
\includegraphics[scale=0.1]{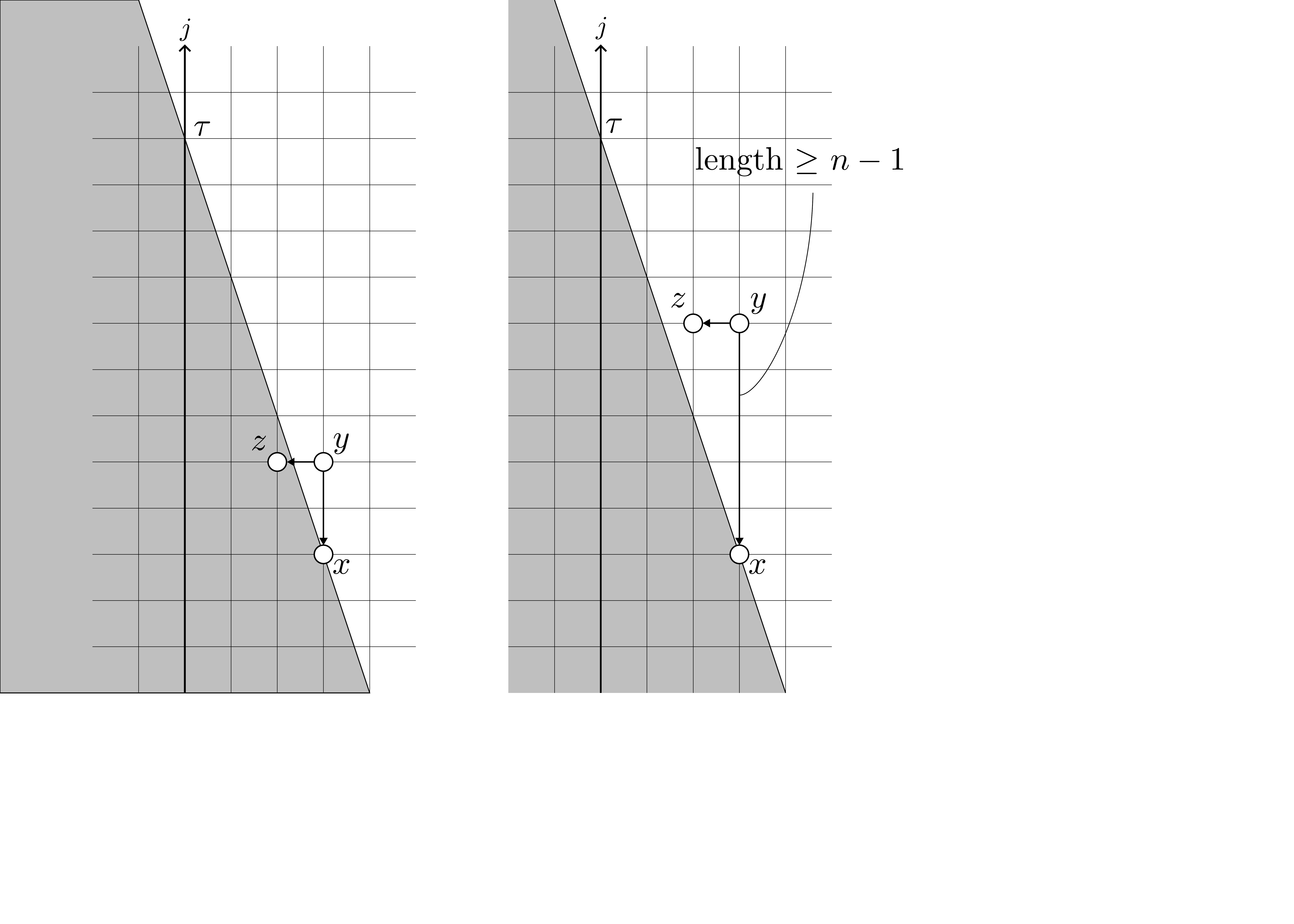}
\caption{Schematic illustration in the proof of Proposition \ref{prop_again}. By Lemma \ref{lem_1}, there is no nontrivial cycle contained in the interior of the gray region. The arrows represent $\partial$. (Left) When the length of $\partial_V(y)$ is less than $n-1$ (here $n-1=3$). Then, $z$ must be in the interior of the gray region, so this does not occur. (Right) This is a possible situation. In this case, the length of $\partial_V (y)$ is at least $n-1$.}
\label{fig_1}
\end{figure}

By the symmetry of $\CFK^{\infty}(K)$ with respect to the $(i,j)$--coordinates, there exist elements $[x',j^x,m]$ and $y'$ in $\CFK^{\infty}(K)$ such that $\partial_H(y') = x'$, and the length of $\partial_H(y')$ is at least $n - 1$. 
This implies that the element $[U^{j^x} x'] \in \mathrm{Tor}(\HFK^-(K))$ has an order of at least $n - 1$. Therefore, we have $\Ord(K)\ge n-1$.
\end{proof}

\begin{remark}
Although all possible cases are considered in the proofs of Lemma \ref{lem_1} and Proposition \ref{prop_main}, we believe that for twist positive knot, only the ``simplest" case in which $\partial b=a+c$, where $[a,0,\tau(K)]$ and $[c,1,\tau(K)-(n-1)]$ are nontrivial cycles and $[b,1,\tau(K)]$ is a base element, actually occurs.
\end{remark}


\end{document}